\documentclass[a4paper,12pt,twoside]{scrartcl}

\usepackage[ansinew]{inputenc}
\usepackage{graphicx} 
\usepackage{textcomp} 
\usepackage{amsfonts} 
\usepackage{amsmath}
\usepackage{amssymb}
\usepackage{latexsym}
\usepackage{caption}
\usepackage{array}
\usepackage{float}
\usepackage{amsthm} 
\usepackage{calc}
\usepackage{perpage} 
\usepackage{textcomp}
\usepackage{verbatim} 
\usepackage{fancyhdr} 

\usepackage[english]{babel}

\usepackage{hyperref}

\DeclareMathOperator*{\argmin}{arg\,min}

\MakePerPage{footnote} 
\DeclareMathOperator{\sconv}{sconv}

\begin{document}
\renewcommand{\oddsidemargin}{0cm}
\renewcommand{\evensidemargin}{0cm}

\begin{titlepage}
\newlength{\centeroffset}
\setlength{\centeroffset}{-0.5\oddsidemargin}
\addtolength{\centeroffset}{0.5\evensidemargin}
\thispagestyle{empty}
\vspace*{\stretch{1}}
\noindent\hspace*{\centeroffset}\makebox[0pt][l]{\begin{minipage}{\textwidth}
\includegraphics[height=1.5cm]{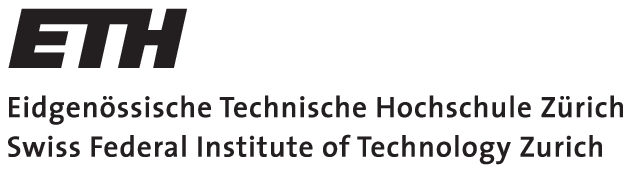}

\vspace{\stretch{1}}
\begin{center}
\Large\bfseries 
\vspace{5mm} 
\normalsize 
\vspace{1cm}
\Huge\bfseries Bounds for Rademacher Processes\\
via Chaining\\
\noindent\rule[-1ex]{\textwidth}{2pt}\\[2.5ex]
\bfseries\normalsize Technical Report
\end{center}
\end{minipage}}
\vspace{\stretch{0.5}}
\vspace{2mm}

\vspace{\stretch{1}}
\noindent\hspace*{\centeroffset}\makebox[0pt][l]{\begin{minipage}{\textwidth}
\centering
\vspace{1cm}
{\bfseries\Large Johannes Christof Lederer}\\
lederer@stat.math.ethz.ch\\
\end{minipage}}

\vspace{1cm}

\begin{abstract}
\textbf{Abstract:}
We study Rademacher processes where the coefficients are functions
evaluated at fixed, but arbitrary covariables. Specifically, we assume the
function class under consideration to be parametrized by the standard
cocube in \itshape l \normalfont dimensions and we are mainly interested in
the high-dimensional, asymptotic situation, that is, \itshape l \normalfont
as well the number of Rademacher variables \itshape n \normalfont go to
infinity with \itshape l \normalfont much larger than \itshape
n\normalfont. We refine and apply classical entropy bounds and Majorizing
Measures, both going back to the well known idea of chaining. That way, we
derive general upper bounds for Rademacher processes. In the linear case and
under high correlations, we further improve on these bounds. In particular,
we give bounds independent of \itshape l \normalfont for highly correlated
covariables.
\end{abstract}

\vspace{\stretch{2}}
\noindent\hspace*{\centeroffset}\makebox[0pt][l]{\begin{minipage}{\textwidth}
\centering
{\bfseries October 2010}\\
\end{minipage}}

\end{titlepage}

\newcommand{\degree}{\ensuremath{^\circ}}
\newcommand{\pmes}{\mathcal{P}}
\newcommand{\mr}{\mathbb{R}}
\newcommand{\mq}{\mathbb{Q}}
\newcommand{\mn}{\mathbb{N}}
\newcommand{\mz}{\mathbb{Z}}
\newcommand{\me}{\mathbb{E}}

\newcommand{\mca}{\mathcal{A}}
\newcommand{\mcb}{\mathcal{B}}

\newcommand{\mpr}{\mathbb{P}}

\newtheorem{theorem}{Theorem}[section]
\newtheorem{proposition}{Proposition}[section]
\newtheorem{lemma}{Lemma}[section]
\newtheorem{corollary}{Corollary}[section]
\newtheorem{remark}{Remark}[section]
\newtheorem{example}{Example}[section]
\newtheorem{definition}{Definition}[section]

\newenvironment{pro}{\begin{proof}[Proof]}{\end{proof}}

\fancypagestyle{page2}{\renewcommand{\headrulewidth}{0pt}
\fancyfoot[C]{\sffamily\bfseries \thepage}
}
\fancypagestyle{normal}{\renewcommand{\headrulewidth}{0pt}
\fancyhead[CO]{ J.C. Lederer}
\fancyhead[CE]{\leftmark}
\fancyfoot[C]{\sffamily\bfseries \thepage}
}

\newpage

\pagestyle{page2}

\section{Introduction}

We study upper bounds for the quantity
\begin{equation}
\label{eq.quantity}
  \me\sup_{\theta\in\Theta}\left|\sum_{i=1}^n\epsilon_i\phi_\theta(x_i)\right|
\end{equation}
with $\Theta:=\{\theta\in\mr^l:\Vert\theta\Vert_1\leq M \}$, i.i.d. Rademacher
variables $\epsilon_i$ and real valued functions $\phi_\theta$ evaluated at
fixed but arbitrary $x_i$. We are
mainly interested in the high-dimensional, asymptotic situation, i.e., $l\gg n$ and
$n,l\to \infty$ and we treat a general setting, the linear case as well as a
setting involving strongly correlated $x_i$. We show in particular that strong
correlations can lead to better asymptotic bounds.\\

Chaining is the main tool for our investigations. For an arbitrary
process $\{Z_\lambda:\lambda\in \Lambda\}$ it means the following:
instead of studying terms of the form $|Z_\lambda-Z_{\lambda'}|$ for (possibly very distinct) random variables
$Z_\lambda,Z_{\lambda'}$ directly, one applies the triangular inequality
\begin{equation*}
  |Z_\lambda-Z_{\lambda'}|\leq \sum_{j=1}^m |Z_{\lambda_n}-Z_{\lambda_{n-1}}|
\end{equation*}
and studies the increments $|Z_{\lambda_n}-Z_{\lambda_{n-1}}|$, where
$\lambda_n,\lambda_{n-1}\in\Lambda$, ${\lambda_0}={\lambda}$ and
${\lambda_m}={\lambda'}$. Usually, the $Z_{\lambda_0},...,Z_{\lambda_m}$ are
constructed such that $Z_\lambda-Z_{\lambda'}$ can be thought
of as the sum of the small
``chain links'' $Z_{\lambda_n}-Z_{\lambda_{n-1}}$. It's often easier
to control these chain links than to control $Z_\lambda-Z_{\lambda'}$
directly. This approach leads to two general bounds for empirical
processes. On the one hand, there is the
classical ``Entropy Bound'' (see for example \cite{Talagrand05},
\cite{vdVaart00} and references therein). Its integral version as stated in
\cite{vdVaart00} is
introduced and refined at the beginning of the second part. Then, we apply this bound to the problem stated
above where we follow ideas
given in \cite{Carl85} for some entropy calculations. On the other hand, there are "Majorizing
Measures" (see for example \cite{Rhee88}, \cite{Talagrand94} and
\cite{Talagrand96}). They are introduced and applied in the third
part. Majorizing Measures are rather difficult to use,
however, we show that for highly correlated
covariables they can lead to substantially better results.\\

We conclude this section with some notation and the main
results.

\paragraph{Notation:} For a pseudometric space $(S,d)$ with unit ball $B$ we denote the
covering numbers by $N(S,d,\epsilon)$, i.e., $N(S,d,\epsilon)$ is the number of translates of $\epsilon B$ needed to
cover $S$. The logarithm of the covering numbers (as a
function of $\epsilon$) is called entropy. We define similarly $D(S,d,\epsilon)$ as the maximal number of $\epsilon$-separated points in
$S$. Obviously, $N(S,d,\epsilon)\leq D(S,d,\epsilon) \leq
N(S,d,\frac{\epsilon}{2})$. And finally, if the pseudometric is induced by a seminorm, we occasionally write
$N(S,\Vert\cdot\Vert,\epsilon)$ or $D(S,\Vert\cdot\Vert,\epsilon)$.\\

We are mainly interested in the pseudometric space $(\Theta,d)$
with
$d(\theta,\theta'):=\Vert(\phi_\theta(x_1)-\phi_{\theta'}(x_1),...,\phi_\theta(x_n)-\phi_{\theta'}(x_n))^T\Vert_2$,
where $x:=(x_1,...,x_n)\in\mathfrak{X}^n$ for an arbitrary set 
$\mathfrak{X}$ and
$\{\phi_\theta:{\mathfrak{X}}\to\mr:\theta\in\Theta\}$ is a set of
functions and we define
$X_\theta(x):=\sum_{i=1}^n\epsilon_i\phi_\theta(x_i)$ for simplicity. The
choice for the pseudometric $d$ is motivated by the fact that $\{X_\theta(x)
:\theta\in\Theta\}$ is
sub-Gaussian with respect to $d$ due to Hoeffding's
inequality, that is
\begin{equation*}
  \mpr(|X_\theta-X_{\theta'}|>u)\leq 2
  \exp\left(-\frac{u^2}{2d(\theta,\theta')}\right)~\forall
  \theta,\theta'\in\Theta.
\end{equation*}
In other words, the tail behavior is as for Gaussian
processes.

\paragraph{Main Results:} We derive upper bounds for the quantity
\eqref{eq.quantity} under three different sets of assumptions. We are not aware of equally
sharp bounds in the literature.\\

In Section~\ref{sec.mainresult}, we derive a bound under the assumption that $\{\phi_\theta:{\mathfrak{X}}\to\mr:\theta\in\Theta\}$ has a certain
contraction property:
\begin{theorem}
\label{lemma.mainresult}
If there exists a function $
A:\mathfrak{X}^n\to\mr$ fulfilling
\begin{equation}
\label{eq.assumptiondud}
  d(\theta,\theta')\leq\sqrt n
  A(x)\Vert\theta-\theta'\Vert_2~\forall\theta,\theta'\in\Theta
\end{equation}
then there is a universal constant K such that for $\theta_0\in\Theta$ arbitrary
\begin{equation}
\label{eqmainresults1}
  \me\sup_{\theta\in\Theta}|X_\theta(x)|\leq \me|X_{\theta_0}(x)|+ K\sqrt{n\log{(l+1)}}\log(n+1)A(x)M.
\end{equation}
\end{theorem}

In the linear case, the $\log(n+1)$ in \eqref{eqmainresults1} can be
omitted and the contraction property \eqref{eq.assumptiondud}
can be relaxed. This is stated in the following theorem we prove in Section~\ref{sec.lindud}:
\begin{theorem}
\label{lemma.lindud}
Let $ \psi_j:\mathfrak{X}\to\mr$ be arbitrary functions for
$j=1,...,l$. If $ \phi_\theta(x_i)=\sum_{j=1}^l\psi_j(x_i)\theta_j $
and if $A:\mathfrak{X}^n\to\mr$ fulfills
\begin{equation*}
\label{eqfrugalcondition}
  d(\theta,0)\leq\sqrt n
  A(x)M~\forall\theta\in\Theta
\end{equation*}
there is a universal constant $K$ such that
\begin{equation*}
  \me \sup_{\theta\in\Theta}|X_\theta(x)|\leq K \sqrt{n\log(l+1)}A(x)M.
\end{equation*}
\end{theorem}
For strongly correlated covariables, we can improve on these bounds. We show this in Section~\ref{sec.appmaj} with the help of Majorizing
Measures. To state the result, we let $X'\in\mr^{n\times l'}$, $X''\in\mr^{n\times l''}$. Furthermore, we denote the $i$-th row of $X'$ ($X''$ resp.) by
$x'_i~(x''_i$ resp.), the columns
by $y'_i~(y''_i$ resp.) and we set $\theta=(\theta',\theta'')$. We then impose
the usual normalization on the matrices, that is $\Vert y'_i\Vert_2=\Vert
y''_i\Vert_2=\sqrt n$ and state the following result: 
\begin{theorem}
\label{lemma.mainmm} 
 Let $g:\mr^2\to \mr$ be a contraction w.r.t. the Euclidean metric. If
 there are orthogonal matrices $R',R''$ such that for all $i$
 \begin{equation}
\label{eq.ellipsoidalcon}
   \sum_{j=1}^nj\frac{(R'y'_i)^2_j}{n},~
   \sum_{j=1}^nj\frac{(R''y''_i)^2_j}{n}\leq 1
 \end{equation}
then there is  a universal constant $K$ such that for $\theta_0\in\Theta$ arbitrary
  \begin{equation*}\label{eq.majmain}
    \me \sup_{\theta\in
      \Theta}\left|\sum_{i=1}^n\epsilon_ig((x_i')^T\theta',(x_i'')^T\theta'')\right|\leq
    \me
    \left|\sum_{i=1}^n\epsilon_ig((x_i')^T\theta_0',(x_i'')^T\theta_0'')\right|+K\sqrt{n\log(
    n+1)}M.
  \end{equation*}
\end{theorem}
So, the factor $\sqrt{n\log(l+1)}\log(n+1)$ in the bound \eqref{eqmainresults1} can be
replaced by $\sqrt{n\log(n+1)}$ in this case. The required correlation is expressed
by assumption \eqref{eq.ellipsoidalcon}: It means, that the columns of the matrices
$X'$ and $X''$
can be enveloped by small ellipsoids. The matrices $R'$ and $R''$ are the
transformations that bring these ellipsoids on the standard form.

\fancyhf{} 
\pagestyle{normal}


\section{Entropy Bounds}

In this part, we introduce entropy bounds and apply them to Rademacher processes. In the first section, we prove adapted versions of two classical
entropy results. The second and the third sections are devoted to the proofs
of Theorem~\ref{lemma.mainresult} and Theorem~\ref{lemma.lindud} and a simple example. 

\subsection{Refinement of Entropy Bounds}

Here, we introduce slightly modified versions of two classical
entropy bounds for empirical processes (see e.g. \cite{vdVaart00} Theorem
2.2.4 and Corollary 2.2.8). The modification is the lower bound for the integration. For convenience, we
give the proofs in detail, although they follow closely the ones given in
\cite{vdVaart00}.\\

Beforehand, we recall the definition of the Orlicz norm $\Vert X
\Vert_\Psi$ for a non-decreasing and convex function $\Psi$ with $\Psi(0)=0$:
\begin{equation*}
  \Vert X
\Vert_\Psi:=\inf\{ A>0:\me \Psi\left( \frac{|X|}{A} \right)\leq 1 \}.
\end{equation*}
We are then able to formulate and prove an important entropy bound:
\begin{lemma}
\label{lemmaentropy1}
Let $\Psi:\mr\to\mr$ be a convex, non-decreasing and non-constant function
with $\Psi(0)=0$ and 
\begin{equation*}
  \limsup_{x,y\to\infty}\frac{\Psi(x)\Psi(y)}{\Psi(cxy)}<\infty
 \end{equation*}
for a constant $c$. Define $\Psi(\infty):=\infty$,
$\Psi^{-1}(y):=\sup\{x:\Psi(x)\leq y\}$ and assume
$\Psi^{-1}(1)>0$. Furthermore, let $\{X_t:t\in T \}$ be a
stochastic process with
\begin{equation*}
  \Vert X_s-X_t\Vert_\Psi \leq Cd(s,t)~\forall s,t\in T
\end{equation*}
and
\begin{equation}
\label{eq.addcondition}
  |X_s-X_t|\leq \alpha d(s,t)~\forall s,t\in T
\end{equation}
for a pseudometric d and positiv constants $C$ and $\alpha$. Then there are universal functions
$K\equiv K(C,\Psi)$ and $U\equiv U(C,\Psi,\alpha)$ such that for all $0<\eta\leq \delta$
\begin{equation}
\label{eqmaxin}
 \Vert\sup_{d(s,t)\leq\delta}|X_s-X_t| \Vert_\Psi\leq K\left( \int_{\frac{\eta}{U}}^{\frac{\eta}{2}}\Psi^{-1}(D(T,d,\epsilon))d\epsilon+\delta\Psi^{-1}(D^2(T,d,\eta))\right).
\end{equation}
\end{lemma}

Comparing this to \cite{vdVaart00}, note that we introduced the additional
condition \eqref{eq.addcondition}. This is to establish the lower integral
bound in the inequality \eqref{eqmaxin}.
\begin{pro}
We may assume that the covering numbers for
$\epsilon>\frac{\eta}{U}$ and the corresponding integral in \eqref{eqmaxin}
are finite since the inequality is trivial otherwise. We then fix $\eta\in\mr^+$ and $k\in\mn$ and construct
nested sets $T_0\subset T_1\subset...\subset T_{k+1}\subset T$ such that
for every $j\leq k+1$ $T_j$ is maximal w.r.t. $d(s,t)>\eta 2^{-j}$ for all $s,t\in T_j$.\\
According to the definition of covering numbers, it holds that $|T_j|\leq
D(T,d,\eta2^{-j})$. We will assume $U2^{-(k+1)}> 1$ ($U$ will be defined later) and hence finitely many
elements in every set, this will be justified later. Now, we will assign
each point $t_{j+1}\in T_{j+1}$ to a unique point $t_{j}\in T_{j}$ such that
$d(t_{j+1},t_j)\leq \eta 2^{-j}$. In this way, we define for all $t_{k+1}\in T_{k+1}$ chains
$t_{k+1}\mapsto ... \mapsto t_0\in T_0$ and use the notation $c(t_{k+1}):=\{t_{k+1},...,t_0\}$.\\
Let $s_{k+1},t_{k+1}\in T_{k+1}$. We then get for elements of these chains 
\begin{align*}
|(X_{s_{k+1}}-X_{s_0})-(X_{t_{k+1}}-X_{t_0})|&=|\sum_{j=0}^{k}(X_{s_{j+1}}-X_{s_j})-\sum_{j=0}^{k}(X_{t_{j+1}}-X_{t_j})|\\
&\leq
\sum_{j=0}^{k}|X_{s_{j+1}}-X_{s_j}|+\sum_{j=0}^{k}|X_{t_{j+1}}-X_{t_j}|\\
&\leq2\sum_{j=0}^{k}\max\{|X_u-X_v|:u\in T_{j+1},v\in T_{j}\cap
c(u)\}.
\end{align*}
Applying Lemma 2.2.2
of \cite{vdVaart00}, we find a constant K depending on $\Psi$ only such that 
\begin{align*}
 &
 \Vert\max|(X_{s_{k+1}}-X_{s_0})-(X_{t_{k+1}}-X_{t_0})|\Vert_\Psi\\
\leq& 2\sum_{j=0}^{k}\Vert\max\{|X_u-X_v|:u\in T_{j+1},v\in T_{j}\cap
c(u)\}\Vert_{\Psi}\\
\leq& 2K\sum_{j=0}^{k}\Psi^{-1}(|T_{j+1}|)\max\{\Vert X_u-X_v \Vert_\Psi :u\in T_{j+1},v\in T_{j}\cap
c(u)\}\\
\leq& 2KC\sum_{j=1}^{k+1}\Psi^{-1}(D(T,d,\eta2^{-j}))\eta2^{-j+1}\\
\leq& 8KC\int_{\eta2^{-(k+2)}}^{\frac{\eta}{2}}\Psi^{-1}(D(T,d,\epsilon))d\epsilon.
\end{align*}
In the first line, the maximum is taken over all $s_{k+1},t_{k+1}\in
T_{k+1}$ and their associated points in $T_0$. We then note that for $\delta\geq\eta$
\begin{align*}
 &\Vert\max\{|X_s-X_t|:s,t\in
T_{k+1}:d(s,t)\leq\delta\} \Vert_\Psi\\
\leq &\Vert\max\{|(X_s-X_{s_0})-(X_t-X_{t_0})|:s,t\in
T_{k+1}:d(s,t)\leq\delta\} \Vert_\Psi\\
+&\Vert\max\{|X_{s_0}-X_{t_0}|:s_0,t_0\in
T_{0},s,t\in
T_{k+1},s_0\in c(s),t_0\in c(t)  \} \Vert_\Psi.
\end{align*}
The first term on the r.h.s. of the last display is bounded according to what we have done above. The second term may be rewritten using
\begin{align*}
  |X_{s_0}-X_{t_0}|\leq& |(X_{s_0}-X_{s_{k+1}})-(X_{t_0}-X_{t_{k+1}})|+|X_{s_{k+1}}-X_{t_{k+1}}|.
\end{align*}
Here, we assign to each $s_0\in T_{0}$ and each $t_0\in T_{0}$ a fixed
$s_{k+1}\in T_{k+1}$, $t_{k+1}\in T_{k+1}$ respectively, such that $s_0\in
c(s)$ and $t_0\in c(t)$. We demand furthermore, that $d(s_{k+1},t_{k+1})\leq\delta$. This yields together with Lemma~2.2.2 of \cite{vdVaart00}
\begin{align*}
 &\Vert\max\{|X_s-X_t|:{s,t\in T_{k+1},d(s,t)\leq\delta}\} \Vert_\Psi\\ 
\leq&16KC\int_{\eta2^{-(k+2)}}^{\frac{\eta}{2}}\Psi^{-1}(D(T,d,\epsilon))d\epsilon+\Vert\max|X_{s_{k+1}}-X_{t_{k+1}}|\Vert_\Psi\\
\leq&16KC\int_{\eta2^{-(k+2)}}^{\frac{\eta}{2}}\Psi^{-1}(D(T,d,\epsilon))d\epsilon+K\Psi^{-1}(D^2(T,d,\eta))\max\Vert
X_{s_{k+1}}-X_{t_{k+1}}\Vert_\Psi
\\
\leq&16KC\int_{\eta2^{-(k+2)}}^{\frac{\eta}{2}}\Psi^{-1}(D(T,d,\epsilon))d\epsilon+KC\delta\Psi^{-1}(D^2(T,d,\eta)).
\end{align*}
The maximum in the second line is taken as described above. We then note that
\begin{align*}
\Vert\sup_{d(s,t)\leq\delta}|X_s-X_t| \Vert_\Psi=& \Vert\sup_{d(s,t)\leq\delta}|(X_{s}-X_{s^*})-(X_{t}-X_{t^*})+(X_{s^*}-X_{t^*})|
 \Vert_\Psi  \\
\leq& 2\Vert\sup_{s\in T}|
X_{s}-X_{s^*}|\Vert_\Psi+ \Vert \max\{|X_s-X_t|:{s,t\in T_{k+1},d(s,t)\leq 3\delta}\}\Vert_\Psi 
\end{align*}
where we define $s^*:=\argmin_{s'\in T_{k+1}}d(s',s)$ and $t^*:=\argmin_{t'\in
  T_{k+1}}d(t',t)$ and use 
\begin{equation*}
  d(s^*,t^*)\leq d(s^*,s)+d(s,t)+d(t,t^*)\leq 3\delta.
\end{equation*}
We find moreover
\begin{align*}
  \Vert\sup_{s\in T}|
  X_{s}-X_{s^*}|\Vert_\Psi&=\inf\{A>0:\me\Psi(\sup_{s\in T}|
X_{s}-X_{s^*}|/A) \leq 1\}\\
&\leq \frac{\alpha\eta 2^{-(k+1)}}{\Psi^{-1}(1)}.
\end{align*}
We may assume  w.l.o.g. that $T$ is not empty and $C,K>0$. So there is a
$k_0\in\mn$ (depending only on $\Psi$ and $\alpha$) such that
$\frac{\alpha 2^{-(k_0+1)}}{\Psi^{-1}(1)} \leq \frac{K}{4}\Psi^{-1}(1)$. Then, 
\begin{align*}
 \label{eqentropyu} \frac{\alpha\eta 2^{-(k_0+1)}C}{\Psi^{-1}(1)}&\leq KC\frac{\eta}{4}\Psi^{-1}(1)\\
  &\leq KC\frac{\eta}{2}(1-2^{-(k_0+1)})\Psi^{-1}(1)\\
  &\leq KC \int_{\eta2^{-(k_0+2)}}^{\frac{\eta}{2}}\Psi^{-1}(D(T,d,\epsilon))d\epsilon.
\end{align*}
We define $U:=2^{k_0+2}$ to conclude the proof.
\end{pro}

Because we often do not need the generality of Lemma~\ref{lemmaentropy1}, we
derive in the following a result for the important special case of sub-Gaussian processes:

\begin{lemma} \label{lemmaentropy2}Let $\{X_t:t\in T \}$ be a sub-Gaussian process w.r.t. a
  pseudometric d such that
\begin{equation*}
  |X_s-X_t|\leq \alpha d(s,t)~\forall s,t\in T
\end{equation*}
for a constant $\alpha$. Then there exists a function $U\equiv
U(\alpha)$ and a universal constant $K$ such that for all $\delta > 0$ and $t_0\in T$ arbitrary
\begin{equation}
\label{eqmaxin2}
  \me\sup_{t:d(t,t_0)\leq\delta}|X_t| \leq \me |X_{t_0}|+ K \int_{\frac{\delta}{U}}^{\frac{\delta}{2}}\sqrt{\log(1+D(T,d,\epsilon))}d\epsilon.
\end{equation}
\end{lemma}

\begin{pro} We apply Lemma~\ref{lemmaentropy1} to $\Psi(x):=e^{x^2}-1$. The function $\Psi$ is convex and increasing and
$\Psi(0)=0$. It holds that
\begin{equation*}
  \limsup_{x,y\to\infty}\frac{\Psi(x)\Psi(y)}{\Psi(xy)}<\infty
\end{equation*}
and 
\begin{equation*}
  \Vert X_s-X_t\Vert_\Psi\leq \sqrt 6 d(s,t)~\forall s,t\in T.
\end{equation*}
So, the conditions of Lemma~\ref{lemmaentropy1} are met. We then set $\eta=\delta$ in Lemma \ref{lemmaentropy1} and note that 
\begin{equation*}
  \Psi^{-1}(m^2)=\sqrt{\log{(1+m^2)}}\leq \sqrt{\log(1+m)^2}=\sqrt 2 \Psi^{-1}(m).
\end{equation*}
So there is a universal constant $K'$ such that (recall that $U\geq 4$,
cf. proof of Lemma~\ref{lemmaentropy1})
\begin{equation*}
 \Vert\sup_{\substack{s,t:d(s,t)\leq\delta}}|X_s-X_t| \Vert_\Psi\leq K' \int_{\frac{\delta}{U}}^{\frac{\delta}{2}}\sqrt{\log(1+D(T,d,\epsilon))}d\epsilon.
\end{equation*}
Since $\sqrt{\log 2}\cdot\me |X| \leq \Vert X\Vert_\Psi$ for any
random variable $X$, there is a constant K such that 
\begin{equation*}
\label{eqmaxzwischen}
  \me\sup_{\substack{s,t:d(s,t)\leq\delta}}|X_s-X_t| \leq K \int_{\frac{\delta}{U}}^{\frac{\delta}{2}}\sqrt{\log(1+D(T,d,\epsilon))}d\epsilon.
\end{equation*}
We conclude the proof by noting that for any $t_0$
\begin{align*}
 \me\sup_{t:d(t,t_0)\leq\delta}|X_t|-\me|X_{t_0}|
\leq\me\sup_{\substack{s,t:d(s,t)\leq\delta}}|X_s-X_{t}|.
\end{align*}
\end{pro}

\subsection{Proof of Theorem~\ref{lemma.mainresult}}
\label{sec.mainresult}
The proof of Theorem~\ref{lemma.mainresult} has two main ingredients: First,
the entropy bound of Lemma~\ref{lemmaentropy1} and second, some subtle entropy
estimates. For the entropy estimates,  we rely on ideas given in
Lemma~1 of \cite{Carl85}.

\begin{proof}[Proof of Theorem~\ref{lemma.mainresult}]
To simplify the notation, we set $X_\theta:=X_\theta(x)$ and $A:=A(x)$. We
then note that, as a consequence of Hoeffding's inequality, $\{X_\theta:\theta\in\Theta\}$ is
sub-Gaussian with respect to the pseudometric
\begin{equation*}
  d(\theta,\theta'):=\Vert(\phi_\theta(x_1)-\phi_{\theta'}(x_1),...,\phi_\theta(x_n)-\phi_{\theta'}(x_n))^T\Vert_2.
\end{equation*}
We find that 
\begin{equation}
\label{eq.mainsqrtn}
  |X_{\theta}-X_{\theta'}|\leq \sqrt{n} d(\theta,\theta')\leq n A\Vert \theta-\theta'\Vert_2~\forall\theta,\theta'\in\Theta.
\end{equation}
Now, we want to calculate the entropy linked with the stochastic
process and the pseudometric $d$. To this end, we define
\begin{equation*}
  V:=\{e_1,...,e_{2l}\}\subset \mr^l
\end{equation*}
using the notation $(e_i)_j:=\delta_{ij}$ for $i\leq l$, where $\delta_{ij}$ is the Kronecker symbol, and
$e_i:=-e_{2l-i+1}$ for $i>l$. So $\Theta$ is the set
$\{\theta\in\mr^l:\exists\lambda\in\mr^{2l},\Vert\lambda\Vert_1\leq
M,\theta=\sum_{i=1}^{2l}\lambda_ie_i \}$. We then fix a
$\lambda\in\mr^{2l}$ such that $\Vert\lambda\Vert_1\leq M$. Define
independent random variables $Y_1,...,Y_k\in V\cup\vec{0}$ with (following \cite{Carl85})
\begin{equation*}
  \mpr(Y_i=e_j)=\frac{|\lambda_j|}{M}~\forall i=1,...,k,j=1,...,2l  
\end{equation*}
and 
\begin{equation*}
  \mpr(Y_i=\vec 0)=1-\sum_{j=1}^{2l}\frac{|\lambda_j|}{M}.
\end{equation*}
We obtain 
\begin{equation*}
  \me Y_i=\frac{1}{M}\sum_{j=1}^{2l}|\lambda_j|e_j\in\Theta~\forall i.
\end{equation*}
Next, we set $\overline{Y}_k:=\frac{1}{k}\sum_{i=1}^{k}Y_i\in\Theta$. One
may check that 
\begin{align*}
  \me[d(M\overline{Y}_k,M\me Y_1)^2]
\leq\frac{4nA^2M^2}{k}
\end{align*}
using the contraction property \eqref{eq.assumptiondud}.
So, the distance of at least one realization of $M\overline{Y}_k$ to $M\me Y_1$
is smaller or equal to $2\sqrt\frac{n}{k} AM$. For the (at most $\binom{2l+k-1}{k}$) realizations of
$M\overline{Y}_k$ and $M\me
Y_1$ it holds that $\forall \theta\in \Theta~\exists\lambda:\Vert\lambda\Vert_1\leq
M,\theta=M\sum_{j=1}^{2l}\frac{|\lambda_j|}{M}e_j$. Hence, using Stirling's
inequalities, we get
\begin{equation*}
  N\left(\Theta,d,2\sqrt\frac{n}{k}AM\right)\leq\binom{2l+k-1}{k}\leq \left(e+\frac{2el}{k}\right)^k.
\end{equation*}
Therefore,
\begin{equation*}
  N(\Theta,d,\epsilon)\leq \left(e+\frac{el\epsilon^2}{2nM^2A^2}\right)^{\frac{4nM^2A^2}{\epsilon^2}+1}
\end{equation*}
when we choose $k:=\lceil\frac{4nM^2A^2}{\epsilon^2}\rceil$. Consequently,
\begin{equation*}
  D(\Theta,d,\epsilon)\leq \left(e+\frac{el\epsilon^2}{8nM^2A^2}\right)^{\frac{16nM^2A^2}{\epsilon^2}+1}.
\end{equation*}
We may now use Lemma~\ref{lemmaentropy1} and get for a universal constant
$K$ and a constant $U$ depending only on $\sqrt{n}$ (see
condition~\eqref{eq.addcondition}~and inequality~\eqref{eq.mainsqrtn})
\begin{equation*}
 \me\sup_{\theta\in\Theta}|X_\theta|- \me|X_{\theta_0}| \leq K\int_{\frac{\sqrt{n}AM}{U}}^{\sqrt n AM}\sqrt{\log{(1+D(\Theta,d,\epsilon))}}d\epsilon.
\end{equation*}
Regarding the last part of the proof of Lemma~\ref{lemmaentropy1} we find a
universal constant $V$ such that $U=\sqrt{n}V$. The results then follows by
a simple calculation.
\end{proof}

\subsection{Proof of Theorem~\ref{lemma.lindud} and an Example}
\label{sec.lindud}
In the linear case, we can get rid of one of the logarithms. This is
because we can transform the parameter space into a lower
dimensional one. We note that in the proof of this lemma, the lower bounds for the integrals in Lemma~\ref{lemmaentropy1} and
Lemma~\ref{lemmaentropy2} are not necessary. Additionally, no difficult
entropy estimates have to be made.

\begin{proof}[Proof of~Theorem~\ref{lemma.lindud}] Again, we set
  $X_\theta:=X_\theta(x)$ and $A:=A(x)$ and note that
\begin{equation*}
  \sup_{\theta\in\Theta}|X_\theta|=\sup_{\theta\in\Theta}|\theta^Ta|
\end{equation*}
with $a:=(\sum_{i=1}^n\epsilon_i\psi_1(x_i),...,\sum_{i=1}^n\epsilon_i\psi_l(x_i))^T\in\mr^l$. The
map $\theta\to|\theta^Ta|$ attains its maximum on $\Theta$ at $\theta_0$ where
$(\theta_0)_i:=M\delta_{ip}$ with $p$ such that $|a_p|\geq |a_m|$ for all $m=1,...,l$. So we have
\begin{equation*}
  \me\sup_{\theta\in\Theta}|X_\theta|=\me\sup_{\theta\in\Theta'}|X_\theta|
\end{equation*}
for $\Theta':=\{(M,0,...,0)^T,...,(0,...,0,M)^T,(0,...,0)^T\}$. As a consequence of Hoeffding's inequality,
$\{X_\theta:\theta\in\Theta'\}$ is sub-Gaussian with respect to the pseudometric
$d(\theta,\theta'):=\Vert(\phi_{\theta}(x_1)-\phi_{\theta'}(x_1),...,\phi_{\theta}(x_n)-\phi_{\theta'}(x_n))^T\Vert_2$
and it holds for all $\theta,\theta'$ that $d(\theta,0)\leq
\sqrt{n}M$. Hence, according to Lemma~\ref{lemmaentropy1}, we
get for a universal constant $K$
\begin{equation*}
  \me \sup_{\theta\in\Theta}|X_\theta|\leq K\int_{0}^{\sqrt n AM}\sqrt{\log{(1+D(\Theta',d,\epsilon))}}d\epsilon.
\end{equation*}
The result follows then using $D(\Theta',d,\epsilon)\leq|\Theta'|=l+1$.
\end{proof}

Finally, we give a simple application:
\begin{example}
  Let $X\in \mr ^{n\times l}$ be normalized such that
  the columns have Euclidean norm $\sqrt{n}$. Moreover, define
  $\vec{\epsilon}:=(\epsilon_1,...,\epsilon_n)$ with Rademacher variables $\epsilon_i$. Then, for
  $X_\theta:=\vec{\epsilon}~^TX\theta,~\theta\in\Theta=\{\theta\in\mr^l:\Vert\theta\Vert_1\leq M$, there is a universal constant K such that
\begin{equation*}
  \me\sup_{\theta\in\Theta}|X_\theta|\leq K\sqrt{n\log{(l+1)}}M.
\end{equation*}
\end{example}


\section{The Majorizing Measures Bound}

In this part, we recall the Majorizing Measures Bound and some
consequence such as the Ellipsoid Theorem. We then apply these tools to prove Theorem~\ref{lemma.mainmm}.\\

\subsection{Majorizing Measures}

Majorizing Measures are known to work well in situations where we have unit
balls of $p$-convex Banach spaces as index sets (see \cite{Guedon08} for
an example and \cite{Pisier89} or \cite{Lindenstrauss79} for the
definitions of $p$-convexity, $p$-type and related terms). Here, we recall
the most important bounds arising in this scope. For the proofs and more detailed introductions we refer to \cite{Rhee88}, \cite{Talagrand94} and
\cite{Talagrand96}.\\

We begin with a basic definition:
\begin{definition}
  Let (T,$\bar d$) be a metric space and $\beta>0$. We set 
  \begin{equation*}
    \gamma_\beta(T,\bar d):=\inf\left\lbrace\sup_{t\in T}
      \left(\int_0^\infty \epsilon^{\beta-1}\left(\log\frac{1}{\mu(B(\bar d,t,\epsilon))}\right)^{\frac{\beta}{2}}d\epsilon\right)^{\frac{1}{\beta}} \right\rbrace,
  \end{equation*}
where $B(\bar d,t,\epsilon)$ is the ball w.r.t. $\bar d$ around $t$ with radius $\epsilon$
and the infimum is taken over all probability
 measures $\mu$ on the Borel-$\sigma$-algebra of T.
\end{definition} 

We then recall the following bounds:
\begin{lemma}
  \label{lemma.talmm1}
  \text{\normalfont(The Majorizing Measures Bound)}
  Any sub-Gaussian process fulfills 
  \begin{equation*}
\label{eq.mmg1}
    \me \sup_{t\in T} X_t \leq K \gamma_1(T,\bar d)
  \end{equation*}
 for a universal constant $K$. 
\end{lemma}
\begin{lemma}
  \label{lemma.talmm2}
  \text{\normalfont (The Ellipsoid Theorem)}
  Let the metric $\bar d$ be induced by the norm on $l^2(\mn)$. Then, for 
  \begin{equation*}
    E:=\{ (t_i)_{i\geq 1}:\sum_{i\geq 1}\frac{t_i^2}{a_i^2}\leq 1
    \}\subset l^2
  \end{equation*}
  with $(a_i)_{i\geq 1}\in l^2(\mn)$ positive and non-increasing we have
  \begin{equation}
\label{eq.mmg2}   
 \gamma_2(E,\bar d)\leq K \sup_{i\geq
    1}a_i\sqrt i
  \end{equation}
for a universal constant K.
\end{lemma}
Using H\"olders inequality, the bound \eqref{eq.mmg2} may be used to give an upper
bound for $\gamma_1(T,\bar d)$. Finally, it holds that
\begin{lemma}
  \label{lemma.talmm3}
  Consider a metric space (T,$\bar d$) and a subset S of T. Then,
  \begin{equation*}
\gamma_\beta(S,\bar d)\leq 2\gamma_\beta(T,\bar d).    
  \end{equation*}
\end{lemma}

\subsection{Proof of Theorem~\ref{lemma.mainmm}}
\label{sec.appmaj}

Now, we show how the process of Theorem~\ref{lemma.mainmm} can be rewritten
such that the relevant set is an ellipsoid and how the bounds
stated above can then be applied. To find reasonable results, however, we have to assume strong correlation among
the covariables. By this, we mean that the columns of the corresponding
matrices are not too different. Or, more precisely, that the columns regarded as
vectors can be collectively enveloped by a small ellipsoid.\\

At first, we state a well known fact:

\begin{proposition}
\label{lemma.symmetric}
Let $\{X_t:t\in T \}$ be a stochastic process with an arbitrary index set
$T$. Assume that the  $\me \sup_{t\in T}X_t=\me \sup_{t\in T}(-X_t)$. Then,
  \begin{equation*}
     \me\sup_{\substack{t\in T}}|X_t|-\me|X_{t_0}|
\leq 2\me\sup_{\substack{t\in T}}X_t
  \end{equation*}
for $t_0\in T$ arbitrary.
\end{proposition}
\noindent Moreover, we set $\frac{0}{0}:=0$ and we
denote by $\sconv A$ the symmetric convex hull of a set $A$. We are then prepared to give the proof of the theorem:
\begin{proof}[Proof~of~Theorem~\ref{lemma.mainmm}]
  Setting
  \begin{align*}
   T'&:=
M\cdot \sconv\left\{ y'_1,...,y'_{l'} \right\}
\\
T''&:=
M\cdot \sconv\left\{ y''_1,...,y''_{l''} \right\}
  \end{align*}
we obtain
\begin{equation*}
\me \sup_{\theta\in
      \Theta}\left|\sum_{i=1}^n\epsilon_ig((x_i')^T\theta',(x_i'')^T\theta'')\right|\leq
\me \sup_{t\in
      T'\times T''}\left|\sum_{i=1}^n\epsilon_ig(t_i',t_i'')\right|.
\end{equation*}
Next we define $a_i^2 := \frac{4n}{i}\cdot M^2$, $(\Pi'(t))_i:=t_{2i-1}$
and $(\Pi''(t))_i:=t_{2i}$. Furthermore, 
\begin{equation*}
  E:=\{ t\in \mr^{2n}:\sum_{i= 1}^{2n}\frac{t_i^2}{a_i^2}\leq 1\}.
\end{equation*}
Then,
\begin{equation*} 
\me \sup_{\theta\in
      \Theta}\left|\sum_{i=1}^n\epsilon_ig((x_i')^T\theta',(x_i'')^T\theta'')\right|
\leq 
\me \sup_{t\in
      E}\left|\sum_{i=1}^n\epsilon_ig((R'^{-1}\Pi'(t))_i,(R''^{-1}\Pi''(t))_{i})\right|.
\end{equation*}
To simplify the notation, we define
\begin{equation*}
  g_i(t):=g((R'^{-1}\Pi'(t))_i,(R''^{-1}\Pi''(t))_{i})
\end{equation*}
and we note that since $g$ is a contraction 
\begin{equation}\label{eq.contractineq}
\bar  d(t,\tilde{t}):=\sqrt{\sum_{i=1}^n (g_i(t)-g_i(\tilde{t}))^2}\leq \Vert t-\tilde{t}\Vert_2=:d_2(t,\tilde{t}).
\end{equation}
Now, let S be a maximal subset of E such that $
  \bar d(t,\tilde{t})>M$ for all $t, \tilde{t}\in S,t\neq \tilde{t}
$. Consequently, $(S,d_2)$ is a metric space and we have due to Cauchy-Schwarz' inequality
\begin{equation*}
  \sup_{t\in
      E}\left|\sum_{i=1}^n\epsilon_ig_i(t)\right|\leq \sqrt{n}M + \sup_{t\in
      S}\left|\sum_{i=1}^n\epsilon_ig_i(t)\right|.
\end{equation*}
With regard to Proposition~\ref{lemma.symmetric}, the quantity to calculate is
\begin{equation*}
  \me \sup_{t\in
      S}\sum_{i=1}^n\epsilon_ig_i(t).
\end{equation*}
To bound this quantity, we apply Hoeffding's inequality, the contraction
property~\eqref{eq.contractineq} and
Lemma~\ref{lemma.talmm1} to obtain for a universal constant $K$
\begin{equation*}
   \me \sup_{t\in
      S}\sum_{i=1}^n\epsilon_ig_i(t)\leq K\gamma_1(S,d_2).
\end{equation*}
Moreover, $
  d_2^2
(t,0)\leq \sum_{i=1}^{2n} a_i^2\leq 8n^2M^2$, so that we arrive at (using
H\"olders inequality)
\begin{align*}
  &\int_0^\infty \sqrt{\log\frac{1}{\mu(B(d_2,t,\epsilon))}}d\epsilon \\
\leq&\int_0^M \sqrt{\log\frac{1}{\mu(B(d_2,t,\epsilon))}}d\epsilon +
\left(\int_M^{4 n M}\frac{d\epsilon}{\epsilon}\right)^{\frac{1}{2}}
\left(\int_0^{\infty} \epsilon \log\frac{1}{\mu(B(d_2,t,\epsilon))}d\epsilon \right )^{\frac{1}{2}}
\end{align*}
We stress, that the balls are with
respect to the set $S$. Finally, 
\begin{equation*}
  \sqrt 2 \left(\int_0^{\infty}
    \epsilon \log\frac{1}{\mu(B(d_2,t,\epsilon))}d\epsilon \right
  )^{\frac{1}{2}}\geq \int_0^M \sqrt{\log\frac{1}{\mu(B(d_2,t,\epsilon))}}d\epsilon.
\end{equation*}
Thus, the proof can be concluded using Lemma~\ref{lemma.talmm2}
and Lemma~\ref{lemma.talmm3}.
\end{proof}


\section{Conclusion}

Classical entropy bounds have proved to be a simple and useful tool in many
applications. However, Majorizing Measures are a priori more powerful in the
treatment of empirical processes. They are
known to outmatch the classical entropy bounds for unit balls of $p$-convex Banach spaces as index
sets. While this is true, the unit ball of $(\mr^l,\Vert
\cdot\Vert_1)$ is not $p$-convex. So far, we only found reasonable results
with Majorizing Measures by invoking high correlation. The results were in
this case independent of the dimension $l$, which is quite important since we
often assume $l\gg n$.

\subsection*{Acknowledgments}
I thank Sara van de Geer for the excellent support. Furthermore, I thank
Mohamed Hebiri for his interest and some helpful suggestions.

\bibliography{../Thesis/Bib/Literatur}

\end{document}